\documentclass[11pt,twoside]{amsart}
\usepackage{amsmath}
\usepackage{amsthm}
\usepackage{amsfonts}
\usepackage{amssymb}
\usepackage{latexsym}
\usepackage[all]{xy}
\usepackage{mathrsfs}
\usepackage{bbm}
\usepackage{enumerate}
\usepackage{tikz-cd}
\usepackage{tikz}
\usepackage{hyperref}
\numberwithin{equation}{section} \theoremstyle{plain}
\newtheorem*{thm*}{Main Theorem}
\newtheorem{thm}{Theorem}[section]

\newtheorem*{cor*}{Corollary}
\newtheorem{lem}[thm]{Lemma}
\newtheorem*{lem*}{Lemma}
\newtheorem{prop}[thm]{Proposition}
\newtheorem*{prop*}{Proposition}

\newtheorem*{rem*}{Remark}

\newtheorem*{exam*}{Example}
\newtheorem{defn}[thm]{Definition}
\newtheorem*{defn*}{Definition}

\newtheorem*{conj*}{Conjecture}
\newtheorem*{ack*}{ACKNOWLEDGEMENTS}

\DeclareMathOperator{\Irr}{Irr}

\newcommand{\End}{\mbox{\rm End}}

\newcommand{\Gkdim}{\mbox{\rm GKdim}}
\newcommand{\GL}{\mbox{\rm GL}}

\newcommand{\id}{\mbox{\rm id}}

\newcommand{\Nil}{\mbox{\rm Nilrad~}}
\newcommand{\rad}{\mbox{\rm rad~}}
\newcommand{\sign}{\mbox{\rm sign}}
\newcommand{\supp}{\mbox{\rm supp}}
\newcommand{\topp}{\mbox{\rm top}}

\newcommand{\YD}{\mbox{$\mathcal{YD}$}}

\newcommand{\Bb}{\mbox{$\mathcal{B}$}}

\newcommand{\Oo}{\mbox{$\mathcal{O}$}}
\newcommand{\CC}{\mbox{$\mathbb{C}$}}
\newcommand{\DD}{\mbox{$\mathbb{D}$}}
\newcommand{\NN}{\mbox{$\mathbb{N}$}}

\newcommand{\ZZ}{\mbox{$\mathbb{Z}$}}
\newcommand{\kk}{\mbox{$\mathbbm{k}$}}
\begin{document}

\title{    Finite GK-dimensional Nichols Algebras over the Infinite Dihedral Group }
%\thanks{$^\dag$Supported by NSFC 11722016.}
\author{Yongliang Zhang}
\subjclass[2010]{16T05 (primary), 16T25 (secondary)}
\address{Department of Mathematics, Nanjing University,
Nanjing 210093, China}
\email{zhangyongliang87@smail.nju.edu.cn}
\date{}

\begin{abstract}
 We contribute to the classification of Hopf algebras with finite
 Gelfand-Kirillov dimension, GK-dimension for short, through the study of 
 Nichols algebras over $\DD_{\infty}$, the infinite dihedral group. 
 We find all the irreducible Yetter-Drinfeld modules $V$ over $\DD_{\infty}$, and determine which Nichols algebras $\Bb(V)$ of $V$ are finite GK-dimensional.
 \end{abstract}
\keywords{ Nichols algebras, Hopf algebras, Yetter-Drinfeld modules.}
\maketitle

\section{Introduction}
In recent years, Hopf algebras with finite Gelfand-Kirillov dimension (GK-dimension for short) received considerable attention, see \cite{A5,A6,W1,Wu,Z}. By the lifting method appearing
in \cite{A2}, the following problems arised naturally in \cite{A1}:
\begin{enumerate}
\item When is $\dim \Bb(V)$ finite?
\item When is $\Gkdim \Bb(V)$ finite?
\end{enumerate}
For the first problem, see \cite{A3} when $\Bb(V)$ is of diagonal type.  Let $G$ be a group, and $\CC$ the field of all complex numbers. One problem is to find
all the Nichols algebras $\Bb(V)$ with finite dimension 
for any $V\in{ _G^G}\YD$, the Yetter-Drinfeld modules over the group algebra $\kk G$. The cases when $G$ is a finite
simple group were studied in \cite{A4,A8,A9,A10,A11,A12}.
For the second problem, great progress was acheived when 
$G$ is an abelian group, see \cite{A5,A6}.
In this paper we deal with the infinite dihedral group $\mathbb{D}_{\infty}$, and aim to find all the irreducible Yetter-Drinfeld modules $V$ over $\mathbb{D}_{\infty}$. Among the irreducible $V$  we classify those $V$ with $\Gkdim\Bb(V)<\infty$. 
Precisely, we prove
\begin{thm}\label{main thm}
The only Nichols algebras of the finite dimensional irreducible Yetter-Drinfeld
 modules over $\kk \DD_{\infty}$ with finite GK-dimension,
 up to isomorphism, are those in the following list.
 \begin{enumerate}[\rm (1)]
 \item $ \Bb(\Oo_{h^n}, \rho_{\pm 1})$ for $n\in \NN$.
 \item $ \Bb(\Oo_1, S_0^+)$.
 \item $ \Bb(\Oo_1, S_0^-)$.
 \item $\Bb(\Oo_1, S_{\lambda}^+)$.
 \item $\Bb(\Oo_1, S_{\lambda}^-)$.
 \end{enumerate} 
\end{thm}
The proof of Theorem \ref{main thm} is given in  Section \ref{sec:3}, \ref{sec:4}, \ref{sec:5} and Section \ref{sec:6}.
%%%%%%%%
\section{Preliminaries}
\subsection{Notations}
Let $\NN$ and $\ZZ$ be the set of positive integers and the set of 
all the integers, respectively. Let $G$ be a group and $\kk$ an algebraic closed field with characteristic zero. In particular, we take $\kk=\CC$, the field
of all complex numbers. For any $g,h\in G$, we write $g\rhd h=ghg^{-1}$, for the adjoint action of $G$ on itself.
The conjugacy class of $g$ in $G$ will be denoted by $\Oo_g$,
and the centralizer of $g$ is denoted by $G^g=\{ h\in G|hg=gh\}$. The center of an algebra $A$ is denoted by $Z(A)$.
\subsection{Nichols algebras}
Let $V$ be a vector space and $c\in \GL(V\otimes V)$. Then the pair
$(V, c)$ is said to be a braided vector space if $c$ is a solution of 
the braid equation
\begin{equation*}
(c\otimes \id)(\id\otimes c)(c\otimes \id)=(\id\otimes c)(c\otimes \id)(\id\otimes c).
\end{equation*}
We call a braided space $(V, c)$ diagonal type if there exists a matrix $\mathfrak{q}=(q_{ij})_{i,j\in \mathbb{I}}$ with $q_{ij}\in \kk^{\times}$ and $q_{ii}\neq 1$ for any $i,j\in \mathbb{I}$ such that
\begin{equation*}
c(x_i\otimes x_j)=q_{ij}x_j\otimes x_i, i,j\in \mathbb{I}.
\end{equation*}
\begin{defn}\cite[Definition 3]{A1}
Let $H$ be a Hopf algebra with bijective antipode $S$. 
A Yetter-Drinfeld
module over $H$ is a vector space $V$ provided with
\begin{enumerate}[$(1)$]
\item a structure of left $H$-module $\mu: H\otimes V\to V$ and
\item a structure of left $H$-comodule $\delta: V\to H\otimes V$, such that
\end{enumerate}
for all $h\in H$ and $v\in V$, the following compatibility condition holds:
\begin{align*}
\delta(h\cdot v)=h_{(1)}v_{(-1)}S(h_{(3)})\otimes h_{(2)}\cdot v_{(0)}.
\end{align*}
The category of left Yetter-Drinfeld modules is denoted by $_H^H\YD$.
\end{defn}
In particular, if $H=\kk G$ is the group algebra of the group $G$, then a Yetter-Drinfeld module over $H$ is a $G$-graded vector space $M=\oplus_{g\in G}M_g$ provided with a $G$-module structure such that $g\cdot M_h=M_{ghg^{-1}}$.

It can be shown that each Yetter-Drinfeld module $V\in {_H^H\YD}$
is a braided vector space with braiding
\begin{align*}
c(x\otimes y)=x_{(-1)}\cdot y\otimes x_{(0)},\quad x,y\in {_H^H\YD}.
\end{align*}

The category of Yetter-Drinfeld modules over $\kk G$ is denoted by $_G^G\YD$. Let $\Oo\subseteq G$ be a conjugacy class of $G$, then we denote by ${_G^G}\YD(\Oo)$ 
the  subcategory of ${_G^G}\YD$ consisting of all $M\in {_G^G}\YD$ with $M=\bigoplus_{s\in \mathcal{O}}M_s$.
\begin{defn}\cite[Definition 1.4.15]{H1}
Let $g\in G$, and let $V$ be a left $\kk G$-module. Define
\begin{align*}
M(g,V)=\kk G\otimes_{\kk G^g}V
\end{align*}
as an object in $_G^G\YD(\Oo_g)$, where $M(g,V)$ is the induced $\kk G$-module, the $G$-grading is given by 
\begin{align*}
\deg (h\otimes v)=h\rhd g, \quad \text{for all}~h\in G, v\in V,
\end{align*}
and the $\kk G$-comodule structure is 
\begin{align*}
\delta(h\otimes v)=(h\rhd g)\otimes (h\otimes v).
\end{align*}
\end{defn}
Let $V\in {_G^G}\YD$. Let $I(V)$ be the largest coideal of $T(V)$ contained in $\oplus_{n\geq 2}T^n(V)$. The Nichols algebra of $V$ is
defined by $\Bb(V)=T(V)/I(V)$. $\Bb(V)$ is called diagonal type if 
$(V, c)$ is of digonal type.

By the following lemmas and their proofs, we can find all the irreducible Yetter-Drinfeld modules $M(g, V)$
 in $_G^G\YD$, once we have known the corresponding irreducible representations $(\rho,V)$ of  $\kk G^{g}$. The
 corresponding Nichols algebra of $M(\Oo_g, V)$ is denoted by $\Bb(\Oo_{g}, \rho)$ or $\Bb(\Oo_{g}, V)$.
 \begin{lem}\cite[Lemma 1.4.16]{H1}\label{lem.1.4.16}
 Let $g\in G$,  $M\in {_G^G}\YD(\Oo_g)$. Then $M(g, M_g)\to M$ is an isomorphism of Yetter-Drinfeld modules in $_G^G\YD$.
 \end{lem}
\begin{lem}\cite[Corollary 1.4.18]{H1}\label{cor.1.4.18}
Let $\{\Oo_{g_l}|l\in L\}$ be the set of the conjugacy classes of $G$. There is a bijection between the disjoint union of the isomorphism 
classes of the simple left $\kk G^{g_l}$-modules, $l\in L$, and the set of
isomorphism classes of the simple Yetter-Drinfeld modules in $_G^G\YD$.
\end{lem}

%%%%%%%%%%%%%%
\subsection {The infinite dihedral group $\DD_{\infty}$}
As we are familiar,
\begin{equation*}
\begin{aligned}
\DD_{\infty}&=\langle  h,g| g^2=1,ghg=h^{-1}\rangle 
                =\{1,g,h^n,gh^n,h^ng,gh^ng|n\in \NN \},
\end{aligned}
\end{equation*}
where 
\begin{align*}
g^{-1}&=g,& (h^n)^{-1}&=gh^ng,&
(gh^n)^{-1}&=gh^n,&(h^ng)^{-1}&=h^ng,&
(gh^ng)^{-1}&=h^n.
\end{align*}
Consider the conjugacy classes in $\DD_{\infty}$.
\begin{align*}
\Oo_1     &=\{ 1\},\quad
\Oo_{h^n}=\{ h^n,gh^ng\}=\{ h^n,h^{-n}\},\forall n\in \NN,\\
\Oo_g     &
%=\{g,gh^2,gh^4,,\ldots, gh^{2n},\ldots,
  %    h^2g, h^4g,\ldots, h^{2n}g,\ldots  \}
               =\{g,gh^{2n},h^{2n}g|n\in \NN \},\quad
\Oo_{gh} 
%=\{  hg,h^3g,\ldots, h^{2n-1}g,\ldots,
%gh,gh^3,\ldots,gh^{2n-1},\ldots, \}
 =\{h^{2k-1}g,gh^{2k-1}| k\in \NN \}.
\end{align*}
The centralizers of one element in each conjugacy class are 
as follows:
\begin{align*}
&G^1=\DD_{\infty},\quad
G^{g}=\{ x\in G\left|  \right.  xg=gx\}=\{ 1,g\}\cong \ZZ_2,\\
&G^{h^n}=\{1, h^k, gh^kg| k\in \mathbb{N}\}\cong \ZZ,\quad
G^{gh}=\{1, gh\}\cong \ZZ_2.
\end{align*}
Consider all the cosets of $G^{g}$ in $G$. For any $n\in \NN$, 
\begin{align*}
&gG^g=G^g,\quad
h^ngG^g=h^nG^g=\{ h^n,h^ng\},\\
&gh^ngG^g=\{ gh^n,gh^ng\},\quad
gh^nG^g=\{ gh^n,gh^ng\}.
\end{align*}
Therefore, a representative of complete cosets of $G$ over $G^g$ is 
\begin{align*}
\{1,h^n,gh^n| n\in \NN\},
\end{align*}

Similarly, consider all the cosets of $G^{h}$. Representatives of complete cosets of $G$ over $G^h$ and $G^{gh}$ are
\begin{align*}
\{1,g\},\quad \{1,h^n,gh^n|n\in \NN\},
\end{align*}
respectively.
%%%%%%%%%%
\subsection{Gelfand-Kirillov dimension}
The Gelfand-Kirillov dimension, GK-dimension for short, 
becomes a powerful tool to study  noncommutative algebras,
especially for  those with infinite dimensions. 
For the definition and properties of the GK-dimension we refer to \cite{G1}.
\begin{defn}\cite{G1}
The Gelfand-Kirillov dimension of a $\kk$-algebra $A$ is 
\begin{align*}
\Gkdim(A)=\sup_V\varlimsup\log_n \dim(V^n),
\end{align*}
where the supremum is taken over all finite dimensional subspaces $V$
of $A$.
\end{defn}
For the GK-dimension of the tensor product of two $\kk$-algebras $A$ and $B$, \cite[Lemma 3.10]{G1} gives $\Gkdim (A\otimes_{\mathbbm{k}}B ) \leq\Gkdim(A)+\Gkdim(B)$. In particular, if $A$ is a left $H$-module algebra, then $\Gkdim (A\# H)\leq \Gkdim(A)+\Gkdim(H)$,
where $A\# H$ is the smash algebra of $A$ and $H$. Here is a useful result for a  $\Gkdim$-infinite algebra.
\begin{lem}\cite[Theorem 2.6]{A13}\label{lem 2.6}
Let $G$ be a finitely generated group and $M\in {_G^G}\YD$
such that $\Oo=\supp M$ is an infinte conjugacy class. Then
$\Gkdim \Bb(M)\# \kk G=\infty$.
\end{lem}

%%%%%%%%%%%%%%
\section{ The Nichols algebra $\Bb(\Oo_{h^n},\rho)$}\label{sec:3}

\begin{prop}\label{prop:3.1}
For any $n\in \NN$ and any irreducible representation 
$(\rho, V)\in \Irr (G^{h^n})$,  $\Gkdim \Bb(\Oo_{h^n}, \rho)<\infty$ if and only if $(\rho, V)$ is either the trivial representation or the sign representation.
\end{prop}
\begin{proof}
Consider the irreducible Yetter-Drinfeld modules
$\mathbbm{k}G\otimes_{\mathbbm{k}G^{h^n}} V=(1\otimes V)\oplus (g\otimes V)$, where $(\rho,V)\in \Irr(G^{h^n}) $ is an irreducible representation of $\mathbbm{k}G^{h^n}$.  By the theory of groups, $V$ is 1-dimensional, since $G^{h^n}$ is abelian. Let $V=\mathbbm{k}x$. Then the module structure of $M(\Oo_{h^n},\rho)$ is 
\begin{align*}
&g\cdot(1\otimes x) =g\otimes \rho(1)(x), &
&h^n\cdot (1\otimes x)=1\otimes \rho(h^n)(x),\\
&gh^n\cdot (1\otimes x)=g\otimes \rho(h^n)(x),&
&h^ng\cdot(1\otimes x)=g\otimes \rho(h^{-n})(x),\\
&gh^ng\cdot (1\otimes x)=1\otimes \rho(gh^ng)(x),&
&g\cdot(g\otimes x)=1\otimes \rho(1)(x),\\
&h^n\cdot (g\otimes x)=g\otimes\rho(h^{-n})(x),&
&gh^n\cdot (g\otimes x)=1\otimes \rho(gh^ng)(x),\\
&h^ng\cdot(g\otimes x)=1\otimes\rho(h^n)(x),&
&gh^ng\cdot (g\otimes x)=g\otimes \rho(h^n)(x).
\end{align*}
The comodule structure 
$\delta: M(\Oo_{h^n},\rho)\to \mathbbm{k}G\otimes M(\Oo_{h^n},\rho)$ of $ M(\Oo_{h^n},\rho)$ is 
\begin{align*}
\delta(1\otimes x)&=(1\rhd h)\otimes(1\otimes x)= h\otimes(1\otimes x), \\
\delta(g\otimes x)&=(g\rhd h)\otimes (g\otimes x)=h^{-1}\otimes (g\otimes x).
\end{align*}
Then $\mathbbm{k}G\otimes_{\mathbbm{k}G^{h^n}} V$ is a Yetter-Drinfeld module over $G$.

Now we will compute the GK-dimension of the Nichols algebra of  $ M(\Oo_{h^n},\rho)$.

First consider the case of  the trivial representation
 $(\epsilon,V)$ of $G^{h^n}$. That is, the module structure
of  $M(\Oo_{h^n},\epsilon)$ is trivial.
The braiding of $M(\Oo_{h^{n}},\rho)$ is given as follows. Write $x_1=1\otimes x$, $ x_2=g\otimes x$. Then we have
\begin{align*}
c(x_1\otimes x_1)
&=x_1\otimes x_1,&
c(x_1\otimes x_2)
&=x_2\otimes x_1,\\
c(x_2\otimes x_1)&=x_1\otimes x_2,&
c(x_2\otimes x_2)&=x_2\otimes x_2.
\end{align*}
The braiding matrix is 
\begin{align*}
\left[\begin{array}{lllll}
1&1\\
1&1
\end{array}\right].
\end{align*}
Therefore,  $\Bb(\Oo_{h^n},\epsilon)\cong S(W)$, the symmetric algebra over $W$, by \cite[Example 31]{A1}, which has GK-dimension 2.

In general, for any irreducible representation of $G^{h^n}$,
we have $h\cdot x=ax$,
for some $a\in \mathbbm{k}^{\times}$. Write $\rho_a$ for the
representation. Therefore, the braiding of $M(\Oo_{h^n},\rho_a)$ is
\begin{align*}
c(x_1\otimes x_1)&
%=x_1^{(-1)}\cdot x_1 \otimes x_{1}^{(0)}=h\cdot x_1\otimes x_1
=ax_1\otimes x_1,
&c(x_1\otimes x_2)
%=x_1^{(-1)}\cdot x_2\otimes x_1^{(0)}
&=a^{-1}x_2\otimes x_1,\\
c(x_2\otimes x_1)&=a^{-1}x_1\otimes x_2,
&c(x_2\otimes x_2)&=ax_2\otimes x_2.
\end{align*}
The braiding matrix is 
\begin{align*}
\left[\begin{array}{lllll}
a&a^{-1}\\
a^{-1}&a
\end{array}\right].
\end{align*}
If $a=-1$, then we see that $M(\Oo_{h^n},\rho_a)$ is of Cartan type with
Cartan matrix
\begin{align*}
\left[\begin{array}{lllll}
2&0\\
0&2
\end{array}\right],
\end{align*}
which is of finite type. Therefore  $\Gkdim\Bb(\Oo_{h^n},\rho_a)<\infty$ by \cite[Theorem 1]{H0}.
If $a^2\neq 1$, then the corresponding Dynkin diagram is
\begin{tikzpicture}[baseline=-1pt]
\node (P0) at (0,0) {$\circ$};
\node (P1) at (2,0) {$\circ$};
\node (P2) at (0,0.3) {$a$};
\node (P3) at (2,0.3) {$a$};
\draw (P0) edge[-] node[above]{$a^{-2}$} (P1);
\end{tikzpicture}. By \cite[Theorem 1.2]{A51} and \cite[Remark 1.6]{A13} or by going through the list of \cite{H01}, we have
$\Gkdim\Bb(\Oo_{h^n},\rho_a)=\infty$ for all $a^2\neq 1$.
\end{proof}
%%%%%%%%%%%%%%%%%%%%%
\section{ The Nichols algebra $\Bb(\Oo_g,\rho)$}\label{sec:4}
%%%%%%%%%%%%%%%%%%%%%%%
Since all the irreducible representations of $\mathbb{Z}_2$ are the unit representation and sign representation, we have the corresponding irreducible Yetter-Drinfeld modules $M(\Oo_g,\epsilon)$ and $M(\Oo_g,\sign)$.

Let $X=\{ 1,h^n,gh^n|n\in \NN\}$. Then
\begin{align*}
M(\Oo_g,\rho)=\bigoplus\limits_{y\in X}h_y\otimes \mathbbm{k}x,
\end{align*}
where the degree of each $h_y$ is given by
\begin{align*}
\deg(h_y)=h_y\rhd g.
\end{align*}
%%%%%%%%%%%%%%
\subsection{The Nichols algebra $\Bb(\Oo_g,\sign)$}
The module structure of $M(\Oo_g, \sign)$ is
\begin{align*}
&g\cdot (1\otimes x)=-1\otimes x,&
&h\cdot (1\otimes x)=h\otimes x,\\
&g\cdot (gh\otimes x)=h\otimes x,&
&h\cdot (gh\otimes x)=-1\otimes x,\\
&g\cdot (h^n\otimes x)=gh^n\otimes x,&
&h\cdot (h^n\otimes x)=h^{n+1}\otimes x,\\
&g\cdot (gh^n\otimes x)=h^n\otimes x,&
&h\cdot (gh^n\otimes x)=gh^{n-1}\otimes x, n\geq 2.
\end{align*}
The comodule structure is
\begin{align*}
\delta(1\otimes x)&=g\otimes (1\otimes x),\quad
\delta(h^n\otimes x)=h^{2n}g\otimes (h^n\otimes x),\quad
\delta(gh^n\otimes x)=gh^{2n}\otimes (gh^n\otimes x).
\end{align*}

For any $n\geq 1$, let
\begin{align*}
a_n=h^n\otimes x,\quad
b_n=gh^n\otimes x, \quad
a_0=1\otimes x.
\end{align*}
Then the module structure is as follows:
\begin{align*}
g\cdot a_0&=-a_0,&
g\cdot a_n&=b_n,&
g\cdot b_n&=a_n,\\
h\cdot a_0&=a_1,&
h\cdot a_n&=a_{n+1},&
h\cdot b_n&=
\begin{cases}
b_{n-1},&n\geq 2\\
-a_0,&n=1
\end{cases}.
\end{align*}
The comodule structure is
\begin{align*}
\delta(a_0)=g\otimes a_0,\quad
\delta(a_n)=h^{2n}g\otimes a_n,\quad
\delta(b_n)=gh^{2n}\otimes b_n.
\end{align*}

The braiding of $M(\Oo_g, \sign)$ is
\begin{align*}
c(a_m\otimes a_n)
%=a_m^{(-1)}\cdot a_n\otimes a_m^{(0)}\\
%=h^{2m}g\cdot a_n\otimes a_m\\
=h^{2m}\cdot b_n\otimes a_m.
\end{align*}
From the module structure we obtain that  
$h^{2m}\cdot b_n=b_{n-2m}$ if $2m<n$,
and $h^{2m}\cdot b_n=-a_{2m-n}$ if $2m\geq n$.
Therefore,
\begin{align*}
c(a_m\otimes a_n)=
\begin{cases}
b_{n-2m}\otimes a_m,& 2m<n\\
-a_{2m-n}\otimes a_m,&  2m\geq n
\end{cases}.
\end{align*}
For any $n\geq 1$ and $m\geq 0$,
\begin{align*}
c(a_m\otimes b_n)
%=a_m^{(-1)}\cdot b_n\otimes a_m^{(0)}\\
%=h^{2m}g\cdot b_n\otimes a_m\\
%=h^{2m}\cdot a_n\otimes a_m\\
=a_{2m+n}\otimes a_m.
\end{align*}
For $n\geq 2$ and $m\geq 1$, we have
\begin{align*}
&c(b_m\otimes b_n)
%&=b_m^{(-1)}\cdot b_n\otimes b_m^{(-1)}\\
=gh^{2m}\cdot b_n\otimes b_m
%&=\begin{cases}
%g\cdot b_{n-2m}\otimes b_m,& 2m<n\\
%g\cdot (-a_{2m-n})\otimes b_m,&2m\geq n
%\end{cases}\\
=\begin{cases}
 a_{n-2m}\otimes b_m,& 2m<n\\
-b_{2m-n}\otimes b_m,&2m\geq n
\end{cases},\\
&c(b_m\otimes b_1)
%&=b_m^{(-1)}\cdot b_1\otimes b_m^{(0)}\\
%&=gh^{2m}\cdot b_1\otimes b_m\\
=g\cdot a_{2m-1}\otimes b_m
=-b_{2m-1}\otimes b_m,\\
&c(b_m\otimes a_n)
%&=b_m^{(-1)}\cdot a_n \otimes b_m^{(0)}\\
%&=gh^{2m}\cdot a_n\otimes b_m\\
=g\cdot a_{2m+n}\otimes b_m
=b_{2m+n}\otimes b_m.
\end{align*}
%%%%%%%%%%%%%%%%%
%%%The trivial representation of G^g%
   \subsection{ The Nichols algebra $\Bb(\Oo_g,\epsilon)$}
The module structure of $M(\Oo_g,\epsilon)$ is
\begin{align*}
&g\cdot (1\otimes x)
=1\otimes \rho(g)(x)=1\otimes x,\\
&h\cdot (1\otimes x)=h\otimes \rho(1)(x)=h\otimes x,\\
&g\cdot (h^n\otimes x)
=gh^n\otimes \rho(1)(x)=gh^n\otimes x,\\
&h\cdot (h^n\otimes x)
=h^{n+1}\otimes \rho(1)(x)
=h^{n+1}\otimes x,\\
&g\cdot (gh^n\otimes x)
=h^n\otimes \rho(1)(x)
=h^n\otimes x,\\
&h\cdot (gh\otimes x)
=1\otimes \rho(g)(x)
=1\otimes x,\\
&h\cdot (gh^n\otimes x)
=gh^{n-1}\otimes \rho(1)(x)
=gh^{n-1}\otimes x, n\geq 2.
\end{align*}

The comodule structure is
\begin{align*}
\delta(1\otimes x)&=g\otimes (1\otimes x),\quad
\delta(h^n\otimes x)=h^{2n}g\otimes (h^n\otimes x),\quad
\delta(gh^n\otimes x)=gh^{2n}\otimes (gh^n\otimes x).
\end{align*}

For $n\geq 1$, let
\begin{align*}
a_n=h^n\otimes x,\quad
b_n=gh^n\otimes x, \quad
a_0=1\otimes x.
\end{align*}
Then the action and coaction are
\begin{align*}
&g\cdot a_0=a_0,\quad
g\cdot a_n=b_n,\quad
g\cdot b_n=a_n,\\
&h\cdot a_0=a_1,\quad
h\cdot a_n=a_{n+1},\quad
h\cdot b_n=
\begin{cases}
b_{n-1},&n\geq 2\\
a_0,&n=1
\end{cases},\\
&\delta(a_0)=g\otimes a_0,\quad
\delta(a_n)=h^{2n}g\otimes a_n,\quad
\delta(b_n)=gh^{2n}\otimes b_n.
\end{align*}
The braiding of $M(\Oo_g, \epsilon)$ is as follows:
\begin{align*}
c(a_m\otimes a_n)
%=a_m^{(-1)}\cdot a_n\otimes a_m^{(0)}\\
=h^{2m}g\cdot a_n\otimes a_m
=h^{2m}\cdot b_n\otimes a_m.
\end{align*}
If $2m<n$, then
\begin{align*}
h^{2m}\cdot b_n=b_{n-2m}.
\end{align*}
If $2m\geq n$, then
\begin{align*}
h^{2m}\cdot b_n=a_{2m-n}.
\end{align*}
Therefore, we have
\begin{align*}
c(a_m\otimes a_n)
&=\begin{cases}
 b_{n-2m}\otimes a_m,& 2m<n\\
 a_{2m-n}\otimes a_m,&2m\geq n
\end{cases},\\
c(a_m\otimes b_n)
%=a_m^{(-1)}\cdot b_n\otimes a_m^{(0)}\\
&=h^{2m}g\cdot b_n\otimes a_m
%&=h^{2m}\cdot a_n\otimes a_m\\
=a_{2m+n}\otimes a_m.
\end{align*}
hold for any $n, m\in \NN$.
For any  $n\geq 2$, we have
\begin{align*}
c(b_m\otimes b_n)
%&=b_m^{(-1)}\cdot b_n\otimes b_m^{(-1)}\\
&=gh^{2m}\cdot b_n\otimes b_m
%=\begin{cases}
%g\cdot b_{n-2m}\otimes b_m,& 2m<n\\
%g\cdot (a_{2m-n})\otimes b_m,&2m\geq n
%\end{cases}
=\begin{cases}
 a_{n-2m}\otimes b_m,& 2m<n\\
b_{2m-n}\otimes b_m,&2m\geq n
\end{cases},\\
c(b_m\otimes b_1)
%&=b_m^{(-1)}\cdot b_1\otimes b_m^{(0)}\\
&=gh^{2m}\cdot b_1\otimes b_m
=g\cdot a_{2m-1}\otimes b_m
=b_{2m-1}\otimes b_m,\\
c(b_m\otimes a_n)
%&=b_m^{(-1)}\cdot a_n \otimes b_m^{(0)}\\
&=gh^{2m}\cdot a_n\otimes b_m
=g\cdot a_{2m+n}\otimes b_m
=b_{2m+n}\otimes b_m.
\end{align*}
Clearly, $\dim \Bb(\Oo_{g}, \rho)=\infty$, since the Yetter-Drinfeld modules are of infinite dimension. For the GK-dimension we have
\begin{prop}\label{prop 4.1}
$\Gkdim \Bb(\Oo_{g}, \rho)=\infty$ for $\rho=\sign$ and $\rho=\epsilon$.
\end{prop}
\begin{proof}
Using Lemma \ref{lem 2.6}, let $M=M(\Oo_g, \rho)$. 
Then $\Gkdim \Bb(\Oo_g, \rho)\# \kk \DD_{\infty}=\infty$ since $\supp M(\Oo, \rho)=\Oo_g$ is an infinite conjugacy class.
 But $\Gkdim \kk\DD_{\infty}<\infty$, this implies the $\Gkdim \Bb(\Oo_{g}, \rho)=\infty$, since $\Gkdim \Bb(\Oo_g, \rho)\# \kk \DD_{\infty}\leq \Gkdim \Bb(\Oo_g, \rho)+\Gkdim \kk\DD_{\infty}$.
\end{proof}
%%%%%%%%%%%%%%%%%%

%%%%%%%%%%%%%%%%%%
\section{ The Nichols algebra $\Bb(\Oo_{gh},\rho)$}\label{sec:5}
Since $G^{gh}\cong \mathbb{Z}_2$ has only 2 irreducible 
representations, the unit representation and sign representation, 
we have the irreducible Yetter-Drinfeld modules 
$M(\Oo_{gh},\epsilon)$ and $M(\Oo_{gh},\sign)$.

Let 
$X=\{ 1, h^n,gh^n|n\in \NN\}$. Then
\begin{align*}
M(\Oo_{gh},\sign)
=\bigoplus\limits_{y\in X}h_y\otimes \mathbbm{k}x,
\end{align*}
where $h_y$ is a renumeration of $X$, and
\begin{align*}
\deg(h_y)=h_y\rhd gh.
\end{align*}
The module structure is 
\begin{align*}
&g\cdot (1\otimes x)=h\otimes \rho(gh)(x),\\
&g\cdot (h^n\otimes x)=gh^n\otimes \rho(1)(x),\\
&g\cdot (gh^n\otimes x)=h^n\otimes \rho(1)(x),\\
&h\cdot (1\otimes x)=h\otimes \rho(1)(x),\\
&h\cdot (h^n\otimes x)=h^{n+1}\otimes \rho(1)(x),\\
&h\cdot (gh^{n+1}\otimes x)=gh^{n}\otimes \rho(1)(x), \\
&h\cdot (gh\otimes x)=h\otimes \rho(gh)(x),
\end{align*}
where $n\geq 1$.

The comodule structure is
\begin{align*}
&\delta(1\otimes x)=(1\rhd gh)\otimes (1\otimes x)=gh\otimes (1\otimes x),\\
&\delta(h^n\otimes x)=(h^n\rhd gh)\otimes (h^n\otimes x)
=h^{2n-1}g\otimes (h^n\otimes x),\\
&\delta(gh^n\otimes x)=(gh^n\rhd gh)\otimes (gh^n\otimes x)
=gh^{2n-1}\otimes (gh^n\otimes x).
\end{align*}

%%%%%%%%%%%%%%%%%%
\subsection{ The Yetter-Drinfeld module $M(\Oo_{gh},\sign)$}
Let 
\begin{align*}
a_0=1\otimes x,\quad
a_n=h^n\otimes x,\quad
b_n=gh^n\otimes x.
\end{align*}
Then we obtain the module structure
\begin{align*}
&g\cdot a_0=-a_1,&
&g\cdot a_n=b_n,&
&g\cdot b_n=a_n,\\
%&h\cdot a_0=a_1\\
&h\cdot a_n=a_{n+1},&
&h\cdot b_{n+1}=b_n,&
&h\cdot b_1=-a_1.
\end{align*}
The comodule structure is
\begin{align*}
\delta(a_0)=gh\otimes a_0,\quad
\delta(a_n)=h^{2n-1}g\otimes a_n,\quad
\delta(b_n)=gh^{2n-1}\otimes b_n.
\end{align*}

The braiding structure is
\begin{align*}
c(a_m\otimes a_n)
%&=a_m^{(-1)}\cdot a_n\otimes a_m
%&=h^{2m-1}g\cdot a_n\otimes a_m
&=\begin{cases}
b_{n-2m+1}\otimes a_m& n>2m-1\\
-a_{2m-n}\otimes a_m&n\leq 2m-1
\end{cases},&
c(a_m\otimes b_n)
%&=a_m^{(-1)}\cdot b_n\otimes a_m^{(0)}
&%=h^{2m-1}g\cdot b_n\otimes a_m
=a_{n+1}\otimes a_m,\\
c(b_m\otimes b_n)
%&=b_m^{(-1)}\cdot b_n\otimes b_m^{(0)}
%&=gh^{2m-1}\cdot b_n\otimes b_m 
&=\begin{cases}
a_{n-2m+1}\otimes b_m& n>2m-1\\
-b_{2m-n+1}\otimes b_m& n\leq 2m-1
\end{cases},&
c(a_m\otimes a_0)
%&=a_m^{(-1)}\cdot a_0\otimes a_m^{(0)}
&%=h^{2m-1}g\cdot a_0\otimes a_m
=-a_{2m}\otimes a_m,\\
c(a_m\otimes b_1)
%&=a_m^{(-1)}\cdot b_1\otimes a_m^{(0)}
&%=h^{2m-1}g\cdot b_1\otimes a_m
=a_{2m}\otimes a_m,&
c(b_m\otimes a_n)
%&=b_m^{(-1)}\cdot a_n\otimes b_m^{(0)}
&%=gh^{2m-1}\cdot a_n\otimes b_m 
=b_{n+2m-1}\otimes b_m,\\
c(b_m\otimes a_0)
%&=b_m^{(-1)}\cdot a_0\otimes b_m^{(0)}
&%=gh^{2m-1}\cdot a_0\otimes b_m
=b_{2m-1}\otimes b_m,&
c(b_m\otimes b_1)
%&=b_m^{(-1)}\cdot b_1\otimes b_m^{(0)}
&%=gh^{2m-1}\cdot b_1\otimes b_m
=-b_{2m-1}\otimes b_m,\\
c(a_0\otimes a_0)
%&=a_0^{(-1)}\cdot a_0\otimes a_0^{(0)}
&%=gh\cdot a_0\otimes a_{0}
=b_1\otimes a_0,&
c(a_0\otimes a_n)
%&=a_0^{(-1)}\cdot a_n\otimes a_0^{(0)}
&%=gh\cdot a_n\otimes a_0
=b_{n+1}\otimes a_0,\\
c(a_0\otimes b_n)
%&=a_0^{(-1)}\cdot b_n\otimes a_0^{(0)}
&%=gh\cdot b_n\otimes a_0
=a_{n-1}\otimes a_0,&
c(a_0\otimes b_1)
%&=a_0^{(-1)}\cdot b_1\otimes a_0^{(0)}
&%=gh\cdot b_1\otimes a_0
=-b_1\otimes a_0,\\
c(b_1\otimes a_0)
%&=b_1^{(-1)}\cdot a_0 \otimes  b_1^{(0)}
&%=gh\cdot a_0\otimes b_1
=b_1\otimes b_1,&
c(b_1\otimes a_n)
%&=b_1^{(-1)}\cdot a_n \otimes  b_1^{(0)}
&%=gh\cdot a_n\otimes b_1
=b_{n+1}\otimes b_1,\\
c(b_1\otimes b_n)
%&=b_1^{(-1)}\cdot b_n \otimes  b_1^{(0)}
&%=gh\cdot b_n\otimes b_1
=a_{n-1}\otimes b_1,&
c(b_1\otimes b_1)
%&=b_1^{(-1)}\cdot b_1 \otimes  b_1^{(0)}
&%=gh\cdot b_1\otimes b_1
=-b_1\otimes b_1.
\end{align*}
%%%%%%%%%%%%%%%%%%%%
\subsection{ The Yetter-Drinfeld module $M(\Oo_{gh},\epsilon)$}
As in the case $M(\Oo_{gh},\sign)$, we write 
\begin{align*}
a_0=1\otimes x,\quad
a_n=h^n\otimes x, \quad
b_n=gh^n\otimes x.
\end{align*}
Then we have the action and coaction 
\begin{align*}
&g\cdot a_0=a_1, &
&g\cdot a_n=b_n,&
&g\cdot b_n=a_n,\\
%h\cdot a_0=a_1,\quad
&h\cdot a_n=a_{n+1},&
&h\cdot b_{n+1}=b_n,&
&h\cdot b_1=a_1,\\
&\delta(a_0)=gh\otimes a_0,&
&\delta(a_n)=h^{2n-1}g\otimes a_n,&
&\delta(b_n)=gh^{2n-1}\otimes b_n.
\end{align*}

The braiding structure is
\begin{align*}
&c(a_m\otimes a_n)
%&=a_m^{(-1)}\cdot a_n\otimes a_m\\
%=h^{2m-1}g\cdot a_n\otimes a_m
=\begin{cases}
b_{n-2m+1}\otimes a_m,& n>2m-1\\
a_{2m-n}\otimes a_m,&n\leq 2m-1
\end{cases},&
&c(a_m\otimes b_n)
%&=a_m^{(-1)}\cdot b_n\otimes a_m^{(0)}\\
%=h^{2m-1}g\cdot b_n\otimes a_m
=a_{n+1}\otimes a_m,\\
&c(b_m\otimes b_n)
%&=b_m^{(-1)}\cdot b_n\otimes b_m^{(0)}\\
%=gh^{2m-1}\cdot b_n\otimes b_m 
=\begin{cases}
a_{n-2m+1}\otimes b_m,& n>2m-1\\
-b_{2m-n+1}\otimes b_m,& n\leq 2m-1
\end{cases},&
&c(a_m\otimes a_0)
%&=a_m^{(-1)}\cdot a_0\otimes a_m^{(0)}\\
%=h^{2m-1}g\cdot a_0\otimes a_m
=a_{2m}\otimes a_m,\\%%%
&c(a_m\otimes b_1)
%&=a_m^{(-1)}\cdot b_1\otimes a_m^{(0)}\\
%=h^{2m-1}g\cdot b_1\otimes a_m
=a_{2m}\otimes a_m,&
&c(b_m\otimes a_n)
%&=b_m^{(-1)}\cdot a_n\otimes b_m^{(0)}\\
%=gh^{2m-1}\cdot a_n\otimes b_m 
=b_{n+2m-1}\otimes b_m,\\
&c(b_m\otimes a_0)
%&=b_m^{(-1)}\cdot a_0\otimes b_m^{(0)}\\
%=gh^{2m-1}\cdot a_0\otimes b_m
=b_{2m-1}\otimes b_m,&
&c(b_m\otimes b_1)
%&=b_m^{(-1)}\cdot b_1\otimes b_m^{(0)}\\
%=gh^{2m-1}\cdot b_1\otimes b_m
=b_{2m-1}\otimes b_m,\\
&c(a_0\otimes a_0)
%&=a_0^{(-1)}\cdot a_0\otimes a_0^{(0)}\\
%=gh\cdot a_0\otimes a_{0}
=b_1\otimes a_0,&
&c(a_0\otimes a_n)
%&=a_0^{(-1)}\cdot a_n\otimes a_0^{(0)}\\
%=gh\cdot a_n\otimes a_0
=b_{n+1}\otimes a_0,\\
&c(a_0\otimes b_n)
%&=a_0^{(-1)}\cdot b_n\otimes a_0^{(0)}\\
%=gh\cdot b_n\otimes a_0
=a_{n-1}\otimes a_0,&
&c(a_0\otimes b_1)
%&=a_0^{(-1)}\cdot b_1\otimes a_0^{(0)}\\
%=gh\cdot b_1\otimes a_0
=b_1\otimes a_0,\\
&c(b_1\otimes a_0)
%&=b_1^{(-1)}\cdot a_0 \otimes  b_1^{(0)}\\
%=gh\cdot a_0\otimes b_1
=b_1\otimes b_1,&
&c(b_1\otimes a_n)
%&=b_1^{(-1)}\cdot a_n \otimes  b_1^{(0)}\\
%=gh\cdot a_n\otimes b_1
=b_{n+1}\otimes b_1,\\
&c(b_1\otimes b_n)
%&=b_1^{(-1)}\cdot b_n \otimes  b_1^{(0)}\\
%=gh\cdot b_n\otimes b_1
=a_{n-1}\otimes b_1,&
&c(b_1\otimes b_1)
%&=b_1^{(-1)}\cdot b_1 \otimes  b_1^{(0)}\\
%=gh\cdot b_1\otimes b_1
=b_1\otimes b_1.
\end{align*}
It is easy to see that $\dim M(\Oo_{gh}, \epsilon)=\dim M(\Oo_{gh}, \sign)=\infty$. For the GK-dimension of $\Bb(\Oo_{gh}, \rho)$, we have
\begin{prop}
$\Gkdim \Bb(\Oo_{gh}, \rho)=\infty$ for $\rho=\sign$ and 
$\rho=\epsilon$.
\end{prop}
\begin{proof}
Similar to the proof of Proposition \ref{prop 4.1}.
\end{proof}
%%%%%%%%%%%%%%%%%%

%%%%%%%%%%%%%%%%%%
\section{ The Nichols algebra $\Bb(\Oo_1,\rho)$}\label{sec:6}
To determine the Nichols algebras associated to the conjugacy class $\Oo_1$,
we need to find all the left simple $\DD_{\infty}$-modules. 
Let $S$ be any left simple $\DD_{\infty}$-module. 
Then $\End_{\mathbbm{k} \mathbb{D}_{\infty}}(S)=\kk \id_S$ by Schur's lemma. 
For any $a\in Z(\mathbbm{k}\DD_{\infty})$, the map $f_a: S\to S, s\mapsto a\cdot s$, is a module
map. So $(h+h^{-1})\cdot s=\lambda s$ for some $\lambda\in \kk$, and hence 
the representation 
\begin{equation*}
\begin{tikzcd}
\rho: \kk\DD_{\infty}\arrow{r}&\End(S)
\end{tikzcd}
\end{equation*}
induces a representation
\begin{equation*}
\begin{tikzcd}
\bar{\rho}: \kk\DD_{\infty}/\langle h+h^{-1}-\lambda\rangle\arrow{r}&\End(S).
\end{tikzcd}
\end{equation*}
In other words, every simple left $\kk\DD_{\infty}$-module is a
 simple left $\kk\DD_{\infty}/\langle h+h^{-1}-\lambda\rangle$-module. 
 %%%%%%%%%%%%%%%
\subsection{ Representations of $\mathbbm{k}\DD_{\infty}/\langle h+h^{-1}-\lambda\rangle$}
\begin{lem}
The center of $\mathbbm{k}\DD_{\infty}$ is $\mathbbm{k}[h+h^{-1}]$, and for $\lambda\in \mathbbm{k}$, 
\begin{align*}
\dim \mathbbm{k}\DD_{\infty}/\langle h+h^{-1}-\lambda\rangle<\infty.
\end{align*}
\end{lem}
\begin{proof}
 Let $A_{\lambda}=\mathbbm{k}\DD_{\infty}/\langle h+h^{-1}-\lambda\rangle$. 
 In $A_{\lambda}$, $h+h^{-1}=\lambda$, by direct computation, the following relations hold:
 \begin{align*}
 hg=\lambda g-gh,\quad
 h^2=\lambda h-1,\quad
 h^3=\lambda h^2-h,\quad \ldots,\quad
 h^{n+1}=\lambda h^n-h^{n-1}=0.
 \end{align*}
 Therefore, $h^n$ can be spaned by $1$ and $h$ in $R$. Hence $gh^n$, 
 $h^ng$, and $gh^ng$ can be spaned by $1,g, h,gh$. We see that 
 $\dim A_{\lambda}\leq 4$.
\end{proof}

By the following lemma, we need to find all primitive othogonal idempotents of $A_{\lambda}$.
\begin{lem}\cite[Corollary 5.17]{A}
Suppose that $A_A=e_1A\bigoplus \ldots \bigoplus e_nA$ is a decomposition of $A$ into indecomposable submodules.
Every simple right $A$-module is isomorphic to one of the modules
\begin{align*}
S(1)=\topp e_1A,\quad \ldots,\quad S(n)=\topp e_nA.
\end{align*}

\end{lem}

Now compute the idempotents of $A_{\lambda}$.
Let
\begin{align*}
(x_1+x_2g+x_3h+x_4gh)^2=x_1+x_2g+x_3h+x_4gh.
\end{align*}

By the following lemma, we will find the primitive idempotents of an algebra.
\begin{lem}\cite[Corollary 4.7 ]{A}
An idempotent $e\in A$ is primitive if and only if the algebra $eAe\cong \End eA$ 
has only two idempotents $0$ and $e$, that is, the algebra $eAe$ is local.
\end{lem}
Taking $x_3=0$, $x_2=\pm \dfrac{1}{2}, x_4=0$, we obtain
\begin{align*}
e_1=\dfrac{1}{2}(1+g),\quad
 e_2=\dfrac{1}{2}(1-g),\quad
1=e_1+e_2,
\end{align*}
which is a decomposition of 1.
By direct computation, the following equalities hold:
\begin{align*}
e_1A_{\lambda}=k(g+1)+k(gh+h),\quad
e_2A_{\lambda}=k(g-1)+k(gh-h).
\end{align*}
Since $\dim A_{\lambda}<\infty$, the Jacobson radical $\rad e_1A_{\lambda}=\Nil e_1A_{\lambda}$ is 
the nil ideal, we need to find all nilpotent elements of $e_1A_{\lambda}$ and $e_2A_{\lambda}$. \\
Let $a=1+g, b=h+gh$.
Then we have
\begin{align*}
a^2=2a,\quad
ab=2b,\quad
ba=\lambda a,\quad
b^2=\lambda b.
\end{align*}
\begin{lem}
 $(x_1a+x_2b)^n=(2x_1+\lambda x_2)^{n-1}(x_1a+x_2b)$.
\end{lem}
\begin{proof}
If $n=2$, then
\begin{align*}
(x_1a+x_2b)^2
&=x_1^2 a^2+x_1x_2ab+x_1x_2ba+x_2^2b^2
=2x_1^2a+2x_1x_2b+
\lambda x_1 x_2a+\lambda x_2^2b\\
&=(2x_1^2+\lambda x_1x_2)a
+(\lambda x_2^2+2x_1 x_2)b
=(2x_1+\lambda x_2)x_1a+(2x_1+\lambda x_2)x_2b.
\end{align*}
By induction on $n$,
\begin{align*}
(x_1a+x_2b)^{n+1}
&=(x_1a+x_2b)^{n}(x_1a+x_2b)
%&=(2x_1+\lambda x_2)^{n-1}(x_1a+x_2b)(x_1a+x_2b)
=(2x_1+\lambda x_2)^{n}(x_1a+x_2b).
\end{align*}
\end{proof}
Let $[x_1(g+1)+x_2(gh+h)]^n=0$.
We have $x_1=-\frac{\lambda}{2}x_2$. Therefore, the set of 
all nilpotent elements is 
\begin{align*}
\Nil e_1A_{\lambda}
=\mathbbm{k}(-\dfrac{\lambda}{2}a+b).
\end{align*}
Therefore,
$\rad e_1A_{\lambda}=\mathbbm{k}(-\frac{\lambda}{2}a+b)$.

Let $c=1-g,d=h-gh$. Then
\begin{align*}
c^2=2c,\quad
d^2=\lambda d, \quad
cd=2d,\quad
dc=\lambda c.
\end{align*}

\begin{lem}
$(x_1c+x_2d)^n=(2x_1+\lambda x_2)^{n-1}(x_1c+x_2d)$
\end{lem}
Let $(x_1c+x_2d)^n=0$. We have
$x_1=-\dfrac{\lambda}{2}x_2$.
Therefore, the set of all nilpotent elements is
\begin{align*}
\Nil e_2 A_{\lambda}=\mathbbm{k}(-\dfrac{\lambda }{2}c+d),
\end{align*}
and $\rad e_2 A_{\lambda} 
=\mathbbm{k}(-\frac{\lambda }{2}c+d)$.

Consequently,
\begin{lem}\label{lem6.6}
Let
$e_1=\frac{1}{2}(1+g)$ and $e_2=\frac{1}{2}(1-g)$.
Then the simple right modules of $A_{\lambda}$ are
\begin{align*}
e_1A_{\lambda}/\mathbbm{k}(-\dfrac{\lambda }{2}a+b) 
\quad\text{and}\quad
e_2A_{\lambda}/\mathbbm{k}(-\dfrac{\lambda }{2}c+d),
\end{align*}
where $\lambda\in \kk$, $a=1+g,  b=h+gh, c=1-g, d=h-gh$.\\
In particular, the simple modules of $A_0$ are
$e_1A_0$ and $e_2A_0.$
\end{lem}

%%%%%%%%%%%%%%%%%%%
\subsection{The simple left $\kk \DD_{\infty}$-modules}
We can consider the left simple modules of $\kk\DD_{\infty}$. Let
\begin{align*}
e_1&=\dfrac{1}{2}(1+g),\quad
e_2=\dfrac{1}{2}(1-g).
\end{align*}
Then
\begin{align*}
A_{\lambda}e_1&= \mathbbm{k}(1+g)+\mathbbm{k}(h^{-1}+gh),\quad&
A_{\lambda}e_2=\mathbbm{k}(1-g)+\mathbbm{k}(h^{-1}-gh).
\end{align*}
Let  $a=1+g, b=h^{-1}+gh$, $c=1-g,d=h^{-1}-gh$. 
Then we have
\begin{align*}
a^2&=2a,&
ab&=\lambda a,&
ba&=2b,&
b^2&=\lambda b,\\
c^2&=2c,&
cd&=\lambda c,&
dc&=2d,&
d^2&=\lambda d.
\end{align*}

By the first paragraph of this section and Lemma \ref{lem6.6} we have
\begin{lem}
The simple left $\kk\DD_{\infty}$-modules are
\begin{align*}
A_{\lambda}e_1/\mathbbm{k}(-\dfrac{\lambda }{2}a+b) \quad \text{and}\quad
A_{\lambda}e_2/\mathbbm{k}(-\dfrac{\lambda }{2}c+d).
\end{align*}
Precisely, if $\lambda\neq 0$, then the corresponding simple modules are
\begin{align*}
S_{\lambda}^{+}=\mathbbm{k}a,\quad
S_{\lambda}^{-}=\mathbbm{k}c,
\end{align*}
where the module structures are
\begin{align*}
g\cdot a=a,\quad
h\cdot a=\dfrac{\lambda}{2}a,\quad
g\cdot c=-c,\quad
h\cdot c=\dfrac{\lambda}{2}c.
\end{align*}
If $\lambda=0$, then
\begin{align*}
S_0^{+}&=A_0e_1
             =\mathbbm{k}(1+g)+\mathbbm{k}(h^{-1}+gh),\\
S_0^{-}&=A_0e_2
             =\mathbbm{k}(1-g)+\mathbbm{k}(h^{-1}-gh).
\end{align*}
The module structures are
\begin{align*}
g\cdot a&=a,&
h\cdot a&=-b,&
g\cdot b&=-b,&
h\cdot b&=a,\\
g\cdot c&=-c,&
g\cdot d&=d,&
h\cdot c&=-d,&
h\cdot d&=c.
\end{align*}
\end{lem}

\begin{proof}
If  $\lambda \neq 0$, then the module structures of $S_0^{+}$ and $S_{0}^{-}$ are
\begin{align*}
g\cdot a&=g(1+g)=g+1=a,\\
h\cdot a&=h(1+g)=h+hg=\lambda-h^{-1}+(\lambda g-gh)\\
           &=\lambda (1+g)-(h^{-1}-gh)
            =\lambda a-b
           %&=\lambda a-\dfrac{\lambda}{2}a
          =\dfrac{\lambda}{2}a,\\
g\cdot c&=g(1-g)=g-1=-c,\\
h\cdot c&=h(1-g)=h-hg=(\lambda-h^{-1})-(\lambda-gh)\\
&=\lambda (1-g)-(h^{-1}-gh)
=\lambda c-d
=\dfrac{\lambda}{2}c.
\end{align*}
If $\lambda =0$, then the module structures of $S_0^{+}$ and 
$S_{0}^{-}$ are
\begin{align*}
g\cdot a&=g(1+g) =g+1 =a,\\
h\cdot a&=h(1+g)
             =h+hg
             =(-h^{-1})-gh
             =-b,\\
g\cdot b&=g(h^{-1}+gh)
             =gh^{-1}+h
            =-gh-h^{-1}=-b,\\
h\cdot b&=h(h^{-1}+gh)
             =1+hgh
             =a,\\
g\cdot c&=g(1-g)=g-1=-c\\
g\cdot d&=g(h^{-1}-gh)
             =gh^{-1}-h
             =hg-(-h^{-1})
             =h^{-1}-gh
             =d,\\
h\cdot c&=h(1-g)=h-hg
              =-h^{-1}-(-gh)
              =-d,\\
h\cdot d
&=h(h^{-1}-gh)
             =1-hgh
             =1-g
            =c.
\end{align*}
\end{proof}
\subsection{The irreducible Yetter-Drinfeld modules }
Let $G=\DD_{\infty}$. From Lemma \ref{lem.1.4.16} and Lemma \ref{cor.1.4.18} we
obtain all the irreducible Yetter-Drinfeld modules in ${_G^G\YD(\Oo_1)}$:
\begin{align*}
M(1,S_{0}^{+})&=\mathbbm{k}G\otimes_{\mathbbm{k}G}S_{0}^{+}
=1\otimes \mathbbm{k}a+1\otimes \mathbbm{k}b,\\
M(1,S_{0}^{-})&=\mathbbm{k}G\otimes_{\mathbbm{k}G}S_{0}^{-}
=1\otimes \mathbbm{k}c+1\otimes \mathbbm{k}d,\\
M(1,S_{\lambda}^{+})&=\mathbbm{k}G\otimes_{\mathbbm{k}G}S_{\lambda}^{+}
=1\otimes \mathbbm{k}a,\\
M(1,S_{\lambda}^{-})&=\mathbbm{k}G\otimes_{\mathbbm{k}G}S_{\lambda}^{-}
=1\otimes \mathbbm{k}c.
\end{align*}
Let $v_1=1\otimes a$ and $v_2=1\otimes b$. Then
the $A_{\lambda}$-module structure is
\begin{align*}
&g\cdot v_1=g\cdot(1\otimes a)=1\otimes g\cdot a=1\otimes a=v_1,\\
&g\cdot v_2=g\cdot(1\otimes b)=1\otimes g\cdot b=1\otimes (-b)=-v_2,\\
&h\cdot v_1=h\cdot(1\otimes a)=1\otimes h\cdot a=1\otimes (-b)=-v_2,\\
&h\cdot v_2=h\cdot(1\otimes b)=1\otimes h\cdot b=1\otimes a=v_1.
\end{align*}
The comodule structure is
\begin{align*}
\delta(v_1)&=\delta(1\otimes a)=1\otimes (1\otimes a)=1\otimes v_1,\quad
\delta(v_2)=\delta(1\otimes b)=1\otimes (1\otimes b)=1\otimes v_2.
\end{align*}
The braiding structure is
\begin{align*}
c(v_1\otimes v_1)
%&=v_1^{(-1)}\cdot v_1\otimes v_1^{(0)}\\
&=v_1\otimes v_1,\quad
c(v_1\otimes v_2)
%&=v_1^{(-1)}\cdot v_2\otimes v_1^{(0)}\\
=v_2\otimes v_1,\\
c(v_2\otimes v_1)
%&=v_2^{(-1)}\cdot v_1\otimes v_2^{(0)}\\
&=v_1\otimes v_2,\quad
c(v_2\otimes v_2)
%&=v_2^{(-1)}\cdot v_2\otimes v_2^{(0)}\\
=v_2\otimes v_2.
\end{align*}
Therefore, the braiding matrix is
\begin{align*}
\left[
\begin{array}{cccc}
1&1\\
1&1
\end{array}
\right].
\end{align*}

Hence $\Gkdim B(V)=2$ by \cite[Example 31]{A1}.
%%%%%%%%%%%%

Consider
$\mathbbm{k}G\otimes_{\mathbbm{k}G}S_{0}^{-}=1\otimes \mathbbm{k}c+1\otimes \mathbbm{k}d$.
Let $w_1=1\otimes c$ and $w_2=1\otimes d$. Then
the $A_{\lambda}$-module structure is
\begin{align*}
&g\cdot w_1=g(1\otimes c)=1\otimes g\cdot c=1\otimes (-c)=-w_1,\\
&g\cdot w_2=g(1\otimes d)=1\otimes g\cdot d=1\otimes d=w_2,\\
&h\cdot w_1=h(1\otimes c)=1\otimes h\cdot c=1\otimes(-d)=-w_2,\\
&h\cdot w_2=h(1\otimes d)=1\otimes h\cdot d=1\otimes c=w_1.
\end{align*}
The comodule structure is
\begin{align*}
\delta(w_1)=\delta(1\otimes c)=1\otimes (1\otimes c)=1\otimes w_1,\quad
\delta(w_2)=\delta(1\otimes d)=1\otimes (1\otimes d)=1\otimes w_2.
\end{align*}
By direct computation, the braiding structure is
\begin{align*}
c(w_1\otimes w_1)
   %             &=w_1^{(-1)}\cdot w_1\otimes w_1^{(0)}\\
 &=w_1\otimes w_1,\quad
 c(w_1\otimes w_2)
%&=w_1^{(-1)}\cdot w_2\otimes w_1^{(0)}\\
=w_2\otimes w_1,\\
c(w_2\otimes w_1)
%&=w_2^{(-1)}\cdot w_1\otimes w_2^{(0)}\\
&=w_1\otimes w_2,\quad
c(w_2\otimes w_2)
%&=w_2^{(-1)}\cdot w_2\otimes w_2^{(0)}\\
=w_2\otimes w_2.
\end{align*}
Therefore, the braiding matrix is
$\left[
\begin{array}{cccc}
1&1\\
1&1
\end{array}
\right]$.

Hence $\Gkdim B(\Oo_1, S^{-}_{0})=2$ by \cite[Example 31]{A1}.
%%%%%%%%%%%%%%%%%%%%%

Consider
$\mathbbm{k}G\otimes_{\mathbbm{k}G}S_{\lambda}^{+}
=1\otimes \mathbbm{k}a$.
Let $w=1\otimes a$. Then the $A_{\lambda}$-module structure is
\begin{align*}
&g\cdot w=g\cdot(1\otimes a)
=1\otimes g\cdot a=1\otimes a=w,\\
&h\cdot w=h\cdot (1\otimes a)=1\otimes h\cdot a=1\otimes \dfrac{\lambda}{2}a=\dfrac{\lambda}{2}w.
\end{align*}
The comodule structure is
\begin{align*}
\delta(w)=\delta(1\otimes a)
=1\otimes (1\otimes a)=1\otimes w.
\end{align*}
The braiding structure is
\begin{align*}
c(w\otimes w)=w_{(-1)}\cdot w\otimes w_{(0)}=w\otimes w.
\end{align*}
Therefore, $\Gkdim B(\Oo_1, S^{+}_{\lambda})=1$.

%%%%%%%%%%%%%%%%%%%%%%%%
Now we consider
$\mathbbm{k}G\otimes_{\mathbbm{k}G}S_{\lambda}^{-}
=1\otimes \mathbbm{k}c$.
Let $v=1\otimes c$. Then the $A_{\lambda}$-module structure is
\begin{align*}
&g\cdot v=g\cdot(1\otimes c)
=1\otimes g\cdot c=1\otimes (-c)=-v,\\
&h\cdot v=h\cdot (1\otimes c)=1\otimes h\cdot c=1\otimes \dfrac{\lambda}{2}c=\dfrac{\lambda}{2}v.
\end{align*}
The comodule structure is
\begin{align*}
\delta(v)=\delta(1\otimes c)=1\otimes (1\otimes c)=1\otimes v.
\end{align*}

The braiding structure is
\begin{align*}
c(v\otimes v)=v_{(-1)}\cdot v\otimes v_{(0)}=v\otimes v.
\end{align*}
Therefore, $\Gkdim B(\Oo_1, S^{-}_{\lambda})=1$.

\end{document}